\documentclass[11pt]{amsart}
\usepackage[utf8]{inputenc}
\usepackage[T1]{fontenc}

\usepackage{amsmath,amssymb,amscd,graphicx, color,tensor}

\makeatletter
\@namedef{subjclassname@2020}{2020 Mathematics Subject Classification}
\makeatother

\usepackage{hyperref}
 \hypersetup{%
  colorlinks=true,
  linkcolor=black,
  citecolor=black,
  urlcolor=blue,
  pdftitle={Stationary states of a chemotaxis consumption system with singular sensitivity and inhomogeneous boundary conditions},
  pdfauthor={Ahn, Lankeit},
  pdfkeywords={},
  bookmarksopen=false,
}
\usepackage{colonequals}

\newcommand{\R}{{ \mathbb{R}  }}

\newcommand{\bke}[1]{\left( #1 \right)}

\newcommand{\f}[2]{\frac{#1}{#2}}
\newcommand{\io}{\int_{\Omega}}
\newcommand{\Om}{\Omega}
\newcommand{\Ombar}{\overline{\Omega}}
\newcommand{\jnorm}[2][]{\left\Vert#2\right\Vert_{#1}}
\newcommand{\nn}{\nonumber}
\newcommand{\set}[1]{\{#1\}}
\newcommand{\Lom}[1]{L^{#1}(\Om)}

\newcommand{\uv}[1][α]{\underline{v_{#1}}}
\newcommand{\uz}[1][α]{\underline{z_{#1}}}

\usepackage{newunicodechar}
\newunicodechar{∂}{\partial}
\newunicodechar{ℝ}{\mathbb{R}} 
\newunicodechar{θ}{\theta}
\newunicodechar{α}{\alpha}
\newunicodechar{∞}{\infty}
\newunicodechar{χ}{\chi}
\newunicodechar{β}{\beta}
\newunicodechar{Δ}{\Delta}
\newunicodechar{∇}{\nabla}
\newunicodechar{ε}{\varepsilon}
\newunicodechar{ψ}{\psi}
\newunicodechar{φ}{\varphi}
\newunicodechar{γ}{\gamma}
\newunicodechar{σ}{\sigma}
\newunicodechar{ν}{\nu}
\newunicodechar{ρ}{\rho}
\newunicodechar{ω}{\omega}
\newunicodechar{δ}{\delta}

\everymath{\displaystyle}

\begin{document}

\newtheorem{defn}{Definition}
\newtheorem{lemma}{Lemma}
\newtheorem{proposition}{Proposition}
\newtheorem{theorem}{Theorem}
\newtheorem{assumption}{Assumption}
\newtheorem{cor}{Corollary}
\newtheorem{remark}{Remark}
\numberwithin{equation}{section}

\title[ ]{Stationary states of a chemotaxis consumption system with singular sensitivity and inhomogeneous boundary conditions}

\author{Jaewook Ahn}%
\address{Department of Mathematics, Dongguk University, Seoul, 04620, Republic of Korea\\
ORCID: 0000-0003-3425-9830}
\email{ jaewookahn@dgu.ac.kr}

\author{Johannes Lankeit}
\address{Leibniz Universität Hannover, Institut für Angewandte Mathematik, Welfengarten 1, 30167 Hannover, Germany\\
ORCID:  0000-0002-2563-7759}
\email{lankeit@ifam.uni-hannover.de}

\begin{abstract}
For given total mass $m>0$ we show unique solvability of the stationary chemotaxis-consumption model 
\[
 \begin{cases}
  0= \Delta u - \chi \nabla \cdot (\frac{u}{v} \nabla v)\\
  0= \Delta v - uv\\
  \int_\Omega u = m
 \end{cases}
\]
under no-flux-Dirichlet boundary conditions in bounded smooth domains $\Omega\subset \mathbb{R}^2$ and $\Omega=B_R(0)\subset \mathbb{R}^d$, $d\ge 3$.
 \end{abstract}
  \maketitle
 
 \section{Introduction}
  When chemotaxis directs motion (eg. of bacteria) toward higher concentrations of oxygen (or other sources of nutrient), the mathematical description of this phenomenon results in chemotaxis-consumption models, in a typical form given by 
 \begin{align}\label{general-consumptionsystem}
  u_t &= Δu - ∇\cdot(uS(v)∇v)\\
  v_t &= Δv - uv, \nn
 \end{align}
 where the density of cells is denoted by $u$ and the concentration of oxygen by $v$, and where $S$ is the chemotactic sensitivity. 

 Especially in the wake of experiments showing the emergence of patterns and large-scale structures in suspensions 
 of chemotactically active bacteria of the species \textit{B. subtilis} in drops of water \cite{tuval_etal}, the mathematical analysis of such systems has flourished (see the survey \cite{lanwin_consumptionsurvey}).

 However, it turned out that if combined with no-flux boundary conditions (the usual choice in most of said studies), \eqref{general-consumptionsystem} does not seem to adequately account for any large-time behaviour beyond convergence to constant states. Consequently, different boundary conditions for the signal $v$ were included, in numerical experiments already in \cite{tuval_etal}; in analytical treatments mainly starting from \cite{Braukhoff17}, with the main options being 
 Robin-type conditions (see \cite{Braukhoff17,BL19}) or (as a limiting case, cf. eg. \cite[Prop. 5.1]{BL19}) Dirichlet conditions. 
 
 With these boundary conditions, results on the fully dynamic system \eqref{general-consumptionsystem} are mostly concerned with existence of solutions, in classical, weak or generalized sense \cite{black_win,black_wu,Braukhoff17,braukhoff_tang,LW22,tian_xiang_jde,tian_xiang_jmfm,wang_win_xiang_1,wang_win_xiang_2,wang_win_xiang_3,wu_xiang}. Some convergence results like that in \cite{tian_xiang_jde,wu_xiang} again rely on the absence of oxygen influx, only rarely is convergence to a stationary state proven, see eg. \cite{FLM21,yang2024long}, or \cite{hong_wang} for a one-dimensional result. 

 Nevertheless, the stationary states on their own appear to be worthwhile objects of study. For classical, production-type chemotaxis systems with Neumann boundary conditions, their investigation goes back at least to the 1980's (\cite{schaaf}), for systems such as \eqref{general-consumptionsystem}, results are more recent: 
 If $S(v)=1$, their existence and uniqueness has been shown in \cite{BL19} for Robin and \cite{LW22,LWY20} for Dirichlet boundary conditions or \cite{knosalla_wrobel} if additional logistic source terms are included, and local asymptotic stability results for the stationary solutions from \cite{LW22} have been obtained in \cite{li_li}. Boundary layers arise if the diffusion coefficient in the signal equation vanishes, see \cite{LWY20} in the case of Dirichlet conditions and \cite{hou_dcds} for Robin boundary conditions and under assumptions of radial symmetry.
 
 Related aerotaxis systems with certain bounded and sign-changing functions $S$ have been investigated in \cite{knosalla,knosalla_nadzieja}. 

In this article we study stationary states of \eqref{general-consumptionsystem} with the choice $S(v)=\frac{χ}{v}$, for some constant $χ>0$, as chemotactic sensitivity, which accounts for the Weber-Fechner law of stimulus perception and is an important ingredient of models that are able to yield migrating bands of bacteria (see \cite{keller_segel_traveling}).

For Robin boundary conditions and in a radially symmetric setting, an analytic solution was obtained in \cite{HLW21}, together with a result on boundary layers, which also occur if Dirichlet conditions are imposed on both components, as shown for the radially symmetric setting in \cite{hou_jmfm}. 

\section{Results}
 We consider the  chemotaxis-consumption model 
     \begin{equation}\label{MODEL_0}
 \left\{
\begin{array}{ll}\vspace{1mm}
u_t =\nabla \cdot(\nabla u- \chi u\nabla \log v),\qquad
v_{t}=\Delta v-uv,\qquad\quad& x\in\Omega,\,\, t>0, \\\vspace{1mm}
\displaystyle   (\nabla  u-  \chi u \nabla \log v)\cdot\nu=0,\qquad\,\,\,\,\, v=v^{*},\qquad\quad& x\in\partial\Omega,\,\, t>0,\\
\end{array}
\right. 
\end{equation}
 in a bounded smooth domain   $\Omega\subset \R^{d}$, $d\ge1$. Here, $\nu$ denotes the unit outward normal to $\partial\Omega$,    $0<\chi\in\R$, and $0<v^{*}\in  C^{2+\theta}(\partial \Omega)$ with some $\theta\in(0,1)$.
We establish the existence and uniqueness of the steady state of \eqref{MODEL_0}.
\begin{theorem}\label{THM1}
 Let $\Omega\subset \R^{d}$, $d\ge1$, be a bounded smooth domain. If $d\ge3$ assume that $\Om=B_R(0)$ for some $R>0$. Moreover,  let $0<\chi\in\R$, and $0<v^{*}\in  C^{2+\theta}(\partial \Omega)$ with some $\theta\in(0,1)$.
 For any positive number $m$, the steady state problem
  \begin{equation}\label{SSP}
  \left\{
\begin{array}{ll}\vspace{1mm}
0=\nabla \cdot(\nabla u_{\infty}-  \chi u_{\infty} \nabla \log v_{\infty}   ),\qquad
0=\Delta v_{\infty}-u_{\infty}v_{\infty},\qquad\quad& x\in\Omega, \\\vspace{1mm}
    (\nabla  u_{\infty}-    \chi u_{\infty} \nabla \log v_{\infty}  )\cdot\nu=0,\qquad\,\,\,\,\, v_{\infty}=v^{*},\qquad\quad& x\in\partial\Omega,
\end{array}
\right. 
\end{equation}
admits a positive unique solution $(u_{\infty},v_{\infty})\in(\mathcal{C}^{2}(\overline{\Omega}))^{2}$ satisfying $\int_{\Omega}u_{\infty}=m$. 
\end{theorem}

\begin{remark}
 While obvious for the case of $d\in\set{1,2}$, it seems worth noting that also for the higher-dimensional case, radial symmetry of the solution (i.e. $v^*$ being constant) is not assumed in Theorem~\ref{THM1}.    
 In fact, the condition that $\Om$ be a ball can be weakened: It is sufficient if $\Om$ lies inside a ball that it shares a small part of its boundary with, see \eqref{domain-condition} below.\\
 For other domains, unique solvability can be asserted for sufficiently small values of $m$ (cf.~Lemma~\ref{LEMDEFM}).
    \end{remark}

\section{Stationary states}

To prove  Theorem~\ref{THM1}, we follow the idea of \cite[Lem.~7.1--7.7]{LW22}. First, we 
reduce \eqref{MODEL_0} into a single equation, see \cite[Lemma 4.1]{BL19} or, alternatively, \cite[Sec. 2]{schaaf} for a different approach. 
\begin{lemma}\label{lem:onlyoneequation}
Let $0<v\in \mathcal{C}^{2}(\overline{\Omega})$. If $u\in\mathcal{C}^{2}(\overline{\Omega})$
satisfies
\begin{equation}\label{NCEQ}
 \left\{
\begin{array}{ll}\vspace{1mm}
0=\nabla \cdot(\nabla u- \chi u\nabla \log v), \qquad\quad& x\in\Omega, \\\vspace{1mm}
\displaystyle   (\nabla  u-  \chi u \nabla \log v)\cdot\nu=0,\qquad\quad& x\in\partial\Omega,
\end{array}
\right. 
\end{equation}
then there exists $\alpha\in \R$ such that
\begin{equation}\label{NCRELATION}
u=\alpha v^{\chi}.
\end{equation}
Conversely, if \eqref{NCRELATION} holds for some $0<v\in \mathcal{C}^{2}(\overline{\Omega})$ and  $\alpha\in \R$, then \eqref{NCEQ} is satisfied.
Furthermore, the signs of $\int_{\Omega}u$ and $\alpha$ coincide.
\end{lemma}
\begin{proof}
We show that \eqref{NCEQ} implies \eqref{NCRELATION}. The other direction follows from a straightforward computation.
Since \eqref{NCEQ} is invariant under sign change of $u$ (and the claim is trivial if $u\equiv0$), we may assume that $\Om_+=\set{x\in \Om\mid u(x)>0}$ is nonempty. In $\Om_+$, we have (for any $ε>0$, because $∇\logε=0$)
\begin{align}
 u|∇\log \f{u}{v^{χ}}|^2
 &=u(∇\log u-\f{∇v^{χ}}{v^{χ}})\cdot ∇\log \f{u}{v^{χ}}
 =∇u\cdot∇\log \f{u}{v^{χ}} - \f{u}{v^{χ}} ∇v^{χ}\cdot∇\log\f{u}{v^{χ}}\nn\\
 &=∇u\cdot∇\log \f{u}{εv^{χ}} - χv^{χ}∇\log v\cdot∇\f{u}{v^{χ}}
 \label{eq:rewrite-u-nabla-log}
\end{align}
From \eqref{NCEQ} we may conclude that 
\begin{equation}\label{weakform}
 \io ∇u\cdot ∇φ = χ\io u∇\log v\cdot ∇φ \qquad \text{for all }φ\in W^{1,2}(\Om).
\end{equation}
Letting $ψ_{ε}\colonequals \max\set{1,\f{u}{εv^{χ}}}$, we have that $\log ψ_{ε}\in W^{1,2}(\Om)$ and from \eqref{eq:rewrite-u-nabla-log}, \eqref{weakform} 
we thus obtain from integrating over $\Om_{ε}\colonequals\set{u\ge εv^{χ}}$
\newcommand{\ioe}{\int_{\Om_{ε}}}
\begin{align*}
 \ioe u|∇\log \f{u}{v^{χ}}|^2 &= \ioe ∇u\cdot ∇\log \f{u}{εv^{χ}} - \ioe χv^{χ}∇\log v\cdot∇\f{u}{v^{χ}}\\
 &= \io ∇u\cdot ∇\log ψ_{ε} - \ioe χv^{χ}∇\log v\cdot∇\f{u}{v^{χ}}\\
 &= \io χu∇\log v \cdot∇\log ψ_{ε} - \ioe χv^{χ}∇\log v\cdot∇\f{u}{v^{χ}}\\
 &= \ioe χu∇\log v \cdot∇\log \f{u}{v^{χ}} - \ioe χv^{χ}∇\log v\cdot∇\f{u}{v^{χ}}=0
\end{align*}
which shows that $\f{u}{v^{χ}}$ is constant on (each connected component of) $\Om_{ε}$ and thus on $\Om_+$ (after taking $ε\searrow 0$) and hence $u=αv^{χ}$ for some $α\in ℝ$ -- throughout $\Om$, for reasons of continuity.

\end{proof}
The following two lemmas are concerned with existence of solutions of the single equation to which \eqref{SSP} has been reduced by Lemma~\ref{lem:onlyoneequation}.
\begin{lemma}\label{LEM7}
Let   $\alpha_*\ge0$ and $0<v^{*}\in  C^{2+\theta}(\partial \Omega)$ with some $\theta\in(0,1)$.
For any $α\in[0,α_*]$, $0\le v\in \mathcal{C}^{\theta}(\overline{\Omega})$, 
the Dirichlet problem
\[
  \left\{
\begin{array}{ll}\vspace{1mm}
0=\Delta w-\alpha w v^{\chi},\qquad\quad& x\in\Omega, \\\vspace{1mm}
w=v^{*},\qquad\quad& x\in\partial\Omega,
\end{array}
\right. 
\]
 admits a unique solution $w\in \mathcal{C}^{2+\tilde{\theta}}(\overline{\Omega}) $, $\tilde{\theta}=\theta\cdot \min\{1,χ\}\in(0,\theta]$, and $w$ always satisfies  $0< w \le \displaystyle\max_{\partial\Omega}v^{*}$. Moreover,  there exists $C_{1}=C_1(α_*)>0$ such that  for every $α\in[0,α_*]$ and every $v\in \mathcal{C}^{\theta}(\overline{\Omega})$ with $0\le v\le \displaystyle\max_{\partial\Omega}v^{*}$, 
 \begin{equation}\label{HOLDERBD}
 \|w\|_{\mathcal{C}^{\theta}(\overline{\Omega})}\le C_{1}.
 \end{equation}
\end{lemma}
\begin{proof}
According to  \cite[Thm~6.14]{GT01}, the Dirichlet problem has a unique solution $w\in \mathcal{C}^{2+\tilde{\theta}}(\overline{\Omega}) $, $\tilde{\theta}\in(0,\theta]$.  Then, the maximum principle yields $0< w \le \displaystyle\max_{\partial\Omega}v^{*}$. The H\"older bound \eqref{HOLDERBD} results from the   elliptic regularity theory (eg. \cite[Thm~8.33]{GT01}) and the boundedness of $\alpha w v^{\chi}$.
 \end{proof}
\begin{lemma}\label{LEM8}
Let   $\alpha\ge 0$. The Dirichlet problem 
\begin{equation}\label{LEM8EQ}
  \left\{
\begin{array}{ll}\vspace{1mm}
0=\Delta v-\alpha v^{\chi+1},\qquad\quad& x\in\Omega, \\\vspace{1mm}
v=v^{*},\qquad\quad& x\in\partial\Omega,
\end{array}
\right. 
\end{equation}
 admits a unique solution $v\in \mathcal{C}^{2}(\overline{\Omega}) $, and   $v$ always satisfies $0< v \le \displaystyle\max_{\partial\Omega}v^{*}$ and 
\[
 \|v\|_{\mathcal{C}^{\theta}(\overline{\Omega})}\le C_{1},
\]
 where $C_{1}>0$ is the number from \eqref{HOLDERBD}.  
\end{lemma}
\begin{proof}
As in \cite[Lem.~7.3]{LW22}, the existence result can be deduced if we use    
Schauder's fixed point theorem with Lemma~\ref{LEM7} 
in
$
X=\{f\in \mathcal{C}^{\theta}(\overline{\Omega})\,|\,0\le f\le \displaystyle\max_{\partial\Omega}v^{*},\, \|f\|_{\mathcal{C}^{\theta}(\overline{\Omega})}\le C_{1}\}$.  
The estimates result from Lemma~\ref{LEM7}.
\end{proof}

Over the course of the next lemmas, we will treat the dependence of the solution $v$ on the parameter $α$. A first observation in this regard consists in the following monotonicity statement.

\begin{lemma}\label{LEM4}
Let $v_{\alpha_{1}}$ and $v_{\alpha_{2}}$ be two solutions of \eqref{LEM8EQ} with $\alpha=\alpha_{1}\ge0$ and $\alpha_{2}\ge0$, respectively. If  $\alpha_{1}\ge  \alpha_{2}$, then $v_{\alpha_{1}}\le v_{\alpha_{2}}$.
\end{lemma}
\begin{proof}
Let $D:=\{x\in \Omega\,|\, v_{\alpha_{1}}> v_{\alpha_{2}}\}$ and $\delta v:=v_{\alpha_{1}}-v_{\alpha_{2}}$. Then, 
\[
  \left\{
\begin{array}{ll}\vspace{1mm}
\Delta \delta v=  \alpha_{1}v_{\alpha_{1}}^{\chi+1}-\alpha_{2}v_{\alpha_{2}}^{\chi+1}\ge \alpha_{2}(v_{\alpha_{1}}^{\chi+1}-v_{\alpha_{2}}^{\chi+1})\ge0,\qquad\quad& x\in D, \\\vspace{1mm}
\delta v=0,\qquad\quad& x\in\partial D.
\end{array}
\right. 
\]
By the maximum principle,   $ \delta v\le 0$  in $D$ and thus, $D=\emptyset$. Therefore, $v_{\alpha_{1}}\le v_{\alpha_{2}}$ in $\Omega$.
\end{proof}

\begin{lemma}\label{lem:islipschitz}
 Let $f\in C^2([a,b])$. Then $g\colon [a,b]^2\to ℝ$,
 \[
  g(u,v)=\begin{cases}
          \f{f(u)-f(v)}{u-v},&u\neq v,\\
          f'(u),&u=v,
         \end{cases}
 \]
 is Lipschitz continuous. 
\end{lemma}
\begin{proof}
 The domain $[a,b]^2$ is convex, $g$ is differentiable, and its derivatives are bounded, e.g.
 \[
  \f{∂g}{∂u}(u,u)= f''(u);\quad \f{∂g}{∂u}(u,v)=\f{f'(u)(u-v)-(f(u)-f(v))}{(u-v)^2} \quad (u\neq v), 
 \]
 where boundedness is easily obtained from (Taylor's theorem and) boundedness of $f''$.
\end{proof}

\begin{lemma}\label{lem:twosidedboundsforvalpha}
 Let $α_*\in(0,∞)$, let $v^*\in C^{2+θ}(∂\Omega)$, $v^*>0$. Then there is $C>0$ such that for any $α\in[0,α_*]$, the solution $v_{α}$ of \eqref{LEM8EQ} satisfies 
 \[
  \max v^* \ge v_{α}(x)\ge C \qquad \text{for all } x\in\Ombar 
 \]
\end{lemma}
\begin{proof}
 For any $α\in[0,α_*]$ and $x\in\Ombar$, by Lemma~\ref{LEM8} and Lemma~\ref{LEM4} we have 
 \[
  \max v^* \ge v_{α}(x)\ge v_{α_*}(x)\ge \min_{\Ombar} v_{α_*}>0
 \]
 and we set $C=\min_{\Ombar}v_{α_*}$.
\end{proof}

\begin{lemma}\label{lem:quotientishoelder}
 Let $α_*\in(0,∞)$. There is $C>0$ such that the following holds: Whenever $α_1,α_2\in[0,α_*]$ and $v_{α_1}$ and $v_{α_2}$ are the corresponding solutions of \eqref{LEM8EQ}, the function $z\colon \Ombar\to ℝ$ defined by 
 \[
  z(x)=\begin{cases}
        \frac{v_{α_2}^{χ+1}(x)-v_{α_1}^{χ+1}(x)}{v_{α_2}(x)-v_{α_1}(x)},&v_{α_1}(x)\neq v_{α_2}(x),\\
        (χ+1)v_{α_1}^{χ}(x),&\text{otherwise},
       \end{cases}
 \]
 satisfies $\jnorm[C^{θ}(\Ombar)]{z}\le C$.
\end{lemma}
\begin{proof}
 With $C>0$ from Lemma~\ref{lem:twosidedboundsforvalpha}, we apply Lemma~\ref{lem:islipschitz} to $f\colon [C,\max v^*]\to ℝ$, $f(s)=s^{χ+1}$ and use \eqref{HOLDERBD} and the fact that the concatenation of a Lipschitz continuous function with Hölder continuous functions remains Hölder continuous. 
\end{proof}

\begin{lemma}
Let $\alpha_{1}>0$. The function
$
v_{\alpha_{1}}'= \displaystyle\lim_{\alpha_{2}\rightarrow \alpha_{1}} \frac{v_{\alpha_{2}}-v_{\alpha_{1}}}{\alpha_{2}-\alpha_{1}}
$ exists,  belongs to $\mathcal{C}^{2}{(\overline{\Omega})}$,  and satisfies
\begin{equation}\label{VPRIMEEQ}
  \left\{
\begin{array}{ll}\vspace{1mm}
\Delta v_{\alpha_{1} }'= v_{\alpha_{1} }^{\chi+1}+\alpha_{1}( \chi+1)v_{\alpha_{1} }^{\chi} v_{\alpha_{1} }',\qquad\quad& x\in\Omega, \\\vspace{1mm}
v_{\alpha_{1} }'=0,\qquad\quad& x\in\partial\Omega.
\end{array}
\right. 
\end{equation}
\end{lemma}
\begin{proof}
Note that $\tilde{v}=\frac{v_{\alpha_{2}}-v_{\alpha_{1}}}{\alpha_{2}-\alpha_{1}}$ satisfies
\[
  \left\{
\begin{array}{ll}\vspace{1mm}
\Delta \tilde{ v}=   v_{\alpha_{2}}^{\chi+1}+ \alpha_{1} \frac{v_{\alpha_{2}}^{\chi+1}-v_{\alpha_{1}}^{\chi+1}}{v_{\alpha_{2}}-v_{\alpha_{1}}}\tilde{ v},\qquad\quad& x\in\Omega, \\\vspace{1mm}
\tilde{ v}=0,\qquad\quad& x\in\partial\Omega.
\end{array}
\right. 
\]
We fix $α_*>0$. From Lemma~\ref{lem:quotientishoelder} and elliptic Schauder estimates (\cite[Thm. 6.6]{GT01}) we find $C>0$ such that $\jnorm[C^{2+θ}(\Ombar)]{\tilde{ v}}\le C$ for any such function $\tilde{ v}$ with $α_1,α_2\in[0,α_*]$. By the Arzelà-Ascoli theorem and uniqueness of the solutions of \eqref{VPRIMEEQ}, we find that $\lim_{α_2\to α_1} \f{v_{α_2}-v_{α_1}}{α_2-α_1}$ exists in $C^2(\Ombar)$ and solves \eqref{VPRIMEEQ}.
\end{proof}

\begin{lemma}\label{LEM6}
Let $\alpha>0$. Then, $v_{α}+αχv_{α}'\ge 0$ in $\Omega$ and $v_{α}+αχv_{α}'\not\equiv 0$.
\end{lemma}
\begin{proof}
Lemma~\ref{LEM4} shows the nonpositivity of $\alpha \chi v_{\alpha}'$. Since
\begin{equation}\label{ADDEQ}
  \left\{
\begin{array}{ll}\vspace{1mm}
\Delta (v_{\alpha}+\alpha \chi v_{\alpha }')= \alpha (\chi+1)v_{\alpha}^{\chi}(v_{\alpha}+\alpha \chi v_{\alpha }'),\qquad\quad& x\in\Omega, \\\vspace{1mm}
(v_{\alpha}+\alpha \chi v_{\alpha }')=v^{*},\qquad\quad& x\in\partial\Omega,
\end{array}
\right. 
\end{equation}
 the maximum principle yields $(v_{\alpha}+\alpha \chi v_{\alpha }')\ge0$. Moreover, $(v_{\alpha}+\alpha \chi v_{\alpha }')\not\equiv 0$ due to $v^{*}>0$.
\end{proof}

At this point we can conclude that $α$ and the total mass (of $u$) directly correspond to each other: 

\begin{lemma}\label{LEMDEFM}
There is $m_{\infty}\in(0,\infty]$ such that the map $m: (0,\infty)\rightarrow (0,m_\infty)$ defined by 
\begin{equation}\label{def:m}
m(\alpha)=\int_{\Omega} \alpha v^{\chi}_{\alpha}
\end{equation} is bijective.
\end{lemma} 
\begin{proof}
Due to 
\[
    m'(\alpha)=\int_{\Omega} (v_{\alpha}^{\chi}+\alpha\chi v_{\alpha}^{\chi-1}v_{\alpha}')  =\int_{\Omega} v_{\alpha}^{\chi-1}\bke{ v_{\alpha}+\alpha\chi v_{\alpha}'   }
\]
and Lemma~\ref{LEM6}, we have  $m'(\alpha)>0$. Thus, $m$ is injective.

Next, we prove the surjectivity of $m$. Note from $0\le \int_{\Omega} \alpha v^{\chi}_{\alpha}\le \alpha   (\max_{\partial\Omega}v^{*} )^{\chi}|\Omega|$ that 
\[
\lim_{\alpha\rightarrow 0} m(\alpha)=0.
\]
If we set $m_{\infty}\colonequals \lim_{α\to \infty} m(α)\in(0,∞]$, which exists due to the monotonicity of $m$, surjectivity follows from continuity.
\end{proof}

In order to obtain solvability not only for small bacterial mass, it would be desirable to have that $m_{\infty}=\infty$. 

Note that the reasoning at this point has to become more involved compared to stationary solutions of the chemotaxis--consumption system \eqref{general-consumptionsystem} with nonsingular sensitivity $S(v)=const$, where mass can be computed according to $\tilde{m}(α)=\io αe^{v_{α}}$ and $\lim_{α\to\infty} \tilde{m}(α)=\infty$ is immediate (cf.~\cite[Lemma~3.15]{BL19}).

To this end, for any $α>0$ we introduce $\uv$ as solution of 
\begin{equation}\label{defunderlinedv}
  \left\{
\begin{array}{ll}\vspace{1mm}
 Δ\uv = α\uv^{χ+1}\qquad\quad&\text{in } \Om, \\\vspace{1mm}
 \uv = \uv[]^* \colonequals \inf_{∂\Om} v^*\qquad\quad& \text{in } ∂\Omega.
\end{array}
\right. 
\end{equation}
We moreover abbreviate 
\begin{equation}\label{def:z}
z_{α}=α^{\f1{χ}}\uv, 
\end{equation} 
so that 
\begin{equation}\label{eq:z}
  \left\{
\begin{array}{ll}\vspace{1mm}
 Δz_{α} = z_{α}^{χ+1}\qquad\quad&\text{in } \Om, \\\vspace{1mm}
 z_{α} =\alpha^{\frac{1}{\chi}}\uv[]^*\qquad\quad& \text{in } ∂\Omega.
\end{array}
\right. 
\end{equation}
and let 
\begin{equation}\label{def:mbar}
 \underline{m}(α)=\io α\uv^{χ} = \io z_{α}^χ.
\end{equation}

\begin{lemma}\label{LEM9}
For any $α>0$, there is a unique $\uv\in \mathcal{C}^2(\Ombar)$ solving \eqref{defunderlinedv}.
It satisfies $\underline{v_{\alpha}}\le v_{\alpha}$ in $\Omega$. Moreover,   if  $\alpha_{1}\le  \alpha_{2}$, then $\alpha_{1} \underline{v_{\alpha_{1}}}^{\chi}\le \alpha_{2} \underline{v_{\alpha_{2}}}^{\chi}$ in $\Omega$, so that $\underline{m}$ is increasing.
 \end{lemma}
 \begin{proof}
Lemma~\ref{LEM8} shows the existence and uniqueness of $\underline{v_{\alpha}}\in \mathcal{C}^{2}(\overline{\Omega})$. The comparison principle yields   $\underline{v_{\alpha}}\le v_{\alpha}$ in $\Omega$. 
  
Let $\underline{v_{\alpha_{1}}}$ and $\underline{v_{\alpha_{2}}}$ be two solutions with $\alpha=\alpha_{1}\ge0$ and $α=\alpha_{2}\ge0$, respectively. 
We show 
\begin{equation}\label{COMPZ}
\alpha_{1} \underline{v_{\alpha_{1}}}^{\chi}\le \alpha_{2} \underline{v_{\alpha_{2}}}^{\chi}\,\,\mbox{ in  }\,\,\Omega \quad\mbox{ if } \,\,  \alpha_{1}\le \alpha_{2}.
\end{equation}
We take $z_{α}$ from \eqref{def:z} and let
$D:=\{x\in \Omega\,|\, z_{\alpha_{1}}> z_{\alpha_{2}}\}$ and $\delta z:=z_{\alpha_{1}}-z_{\alpha_{2}}$. Then, by \eqref{eq:z}, 
\[
  \left\{
\begin{array}{ll}\vspace{1mm}
\Delta \delta z=  z_{\alpha_{1}}^{\chi+1}-z_{\alpha_{2}}^{\chi+1}\ge0,\qquad\quad& x\in D, \\\vspace{1mm}
\delta z=(\alpha_{1}^{\frac{1}{\chi}}-\alpha_{2}^{\frac{1}{\chi}})\uv[]^*\le 0,\qquad\quad& x\in\partial D,
\end{array}
\right. 
\]
and thus, as in the proof of Lemma~\ref{LEM4} we have $z_{\alpha_{1}}\le z_{\alpha_{2}}$ in $\Omega$. 
Namely, $\alpha_{1} \underline{v_{\alpha_{1}}}^{\chi}\le \alpha_{2} \underline{v_{\alpha_{2}}}^{\chi}$.
 \end{proof}
 
 \begin{lemma}\label{LEM10a}
If  $d\in \set{1,2}$, then 
\begin{equation}\label{subsolunbdd}
\lim_{\alpha \rightarrow \infty}\underline{m}(\alpha)= \infty.
\end{equation}
 \end{lemma}
\begin{proof}
 From Lemma~\ref{LEM9}, we note that
 $
 \underline{m}'(\alpha)=\frac{d}{d\alpha }\int_{\Omega} z_{\alpha}^{\chi}\ge0$ and hence $M\colonequals\lim_{α\to \infty} \underline{m}(α)\in(0,∞]$ exists.
Suppose \eqref{subsolunbdd} does not hold, i.e., 
$
M=\lim_{\alpha \rightarrow \infty}\underline{m}(\alpha)< \infty$. 
Let $\alpha \ge  \alpha_{1} = 1$. By the maximum principle and Lemma~\ref{LEM9}, 
there exists  $C_{1}>0$ such that 
$  C_1\le z_{\alpha_{1}}(x)\le z_{\alpha}(x) \le \alpha^{\frac{1}{\chi}}\uv[]^*$ for all $x\in\overline{\Omega}$ and all $α\ge 1$. The last of these inequalities shows that $∂_{ν}z_{α}\ge 0$ on $∂\Om$; accordingly, for every $α\ge 1$, 
\begin{align}
 \int_{\Omega }|\nabla \log z_{\alpha}|^{2}\le
\int_{\partial \Omega} \frac{1}{z_{\alpha}}\frac{\partial z_{\alpha}}{\partial \nu}+\int_{\Omega }|\nabla \log z_{\alpha}|^{2}
&=\int_{\Omega }\Delta \log z_{\alpha}+\int_{\Omega }|\nabla \log z_{\alpha}|^{2}\nn\\
& = \io \f{Δz_{α}}{z_{α}}
=\int_{\Omega }z_{\alpha}^{\chi}=\underline{m}(\alpha)\le M.\label{eq:intnablalog}
\end{align}
Moreover, due to the fact that $\log y \le \frac{1}{\chi}y^{\chi}$ for all $y\ge0$, 
\begin{equation}\label{eq:intlog}
 \io \log\f{z_{α}}{C_1}\le \f{1}{χC_1^{χ}}\io z_{α}^{χ}\le \f{M}{χC_1^{χ}}.
\end{equation}
Note also from the Moser-Trudinger inequality that there exist positive constants $C_{2}$, $C_{3}$ satisfying 
\begin{equation}\label{MTINEQ}
\int_{\Omega} e^{|\varphi|}\le C_{2}\exp\{ C_{3}\int_{\Omega} |\nabla \varphi|^{2}+|\varphi|   \}\quad \mbox{ for all}\quad \varphi\in W^{1,2}(\Omega).
\end{equation}
Setting $p=2χ+2$, from \eqref{MTINEQ}, \eqref{eq:intnablalog} and \eqref{eq:intlog} we obtain that for every $α\ge 1$, 
\begin{align*}
\frac{1}{C_{1}^{p} }\int_{\Omega}   z_{\alpha}^{p}=\int_{\Omega} e^{ \log (( {z_{\alpha}/C_{1})^{p})     }}
&\le C_{2}\exp\{ C_{3}\int_{\Omega} |\nabla \log ((  z_{\alpha}/C_{1})^{p})    |^{2}+|\log (( z_{\alpha}/C_{1})^{p})    |   \}\\
&=C_{2}\exp\{ C_{3}p^{2}\int_{\Omega} |\nabla \log   z_{\alpha}     |^{2}+  C_{3}p\int_{\Omega} \log ( z_{\alpha}/C_{1})        \}\\
&\le C_{2}\exp\{ C_{3}  p^{2}M+ C_{3}p \frac{M}{\chi C_{1}^{\chi}}   \}, 
\end{align*}
in particular, there is $C_4>0$ such that $\jnorm[\Lom p]{z_{α}}\le C_4$ for all $α\in[1,\infty)$.
From the Gagliardo-Nirenberg interpolation inequality, we obtain $C_5>0$ such that
\[
\| z_{\alpha}\|_{L^{\infty}(\Omega)}\le C_5\left( \| z_{\alpha}\|_{L^{2\chi+2}(\Omega)}^{\theta} \|D^{2} z_{\alpha}\|_{L^{2}(\Omega)}^{1-\theta} + \| z_{\alpha}\|_{L^{2\chi+2}(\Omega)}\right),\qquad \theta=\frac{(4-d)(2\chi+2)}{(4-d)(2\chi+2)+2d}
\]
for every $α\in[1,\infty)$. Using \eqref{eq:z} and the elliptic regularity result \cite[Lem. 9.17]{GT01}, we also note that with some $C_6>0$, 
\[
\|D^{2} z_{\alpha}\|_{L^{2}(\Omega)}\le \|  z_{\alpha}-\alpha^{\frac{1}{\chi}}\uv[]^*\|_{W^{2,2}(\Omega)} \le C_6 \|  z_{\alpha}^{\chi+1}\|_{L^{2}(\Omega)} = C_6\jnorm[\Lom p]{z_{α}}^{χ+1}
\]
holds for every $α\in[1,\infty)$. 
It follows that 
\[
 \alpha^{\frac{1}{\chi}}\uv[]^* = \jnorm[\Lom{\infty}]{z_{α}} \le C_5(C_4^{\theta}(C_6C_4)^{1-\theta}+C_4)
\]
for every $α\in[1,\infty)$, which for large $α$ 
constitutes a contradiction, thus proving \eqref{subsolunbdd}.
%
%
\end{proof}

Due to its reliance on the Trudinger--Moser inequality, the proof presented for Lemma~\ref{LEM10a} is inherently (at most) two-dimensional.
In higher-dimensional domains, we resort to explicit subsolutions. In order to compute these, let us consider radially symmetric functions on $B_R\subset ℝ^d$.

\begin{lemma}\label{lem:subsolution}
 Let $d\ge 2$, $R>0$, $α>0$. The function 
 \[
  \uz(x)=γ|x|^{β}, \qquad x\in \overline{B_R(0)}\subset ℝ^d, 
 \]
with $β=\alpha^{\frac{1}{2}}(\uv[]^*)^{\frac{\chi}{2}}R$ and $γ=\alpha^{\frac{1}{\chi}}\uv[]^*R^{-β}$ satisfies 
\begin{equation}\label{subsolution}
 -Δ\uz \le - \uz^{χ+1} \quad\text{ in } B_R(0) \qquad \text{ and } \qquad 
 \uz \le α^{\f1{χ}}\underline{v_{\alpha}} \quad \text{ in } \overline{B_R(0)}.
\end{equation}
\end{lemma}
\begin{proof}
 Since, $\uz$ is radially symmetric, for $|x|=r\in(0,R)$ we can compute  
 \begin{align*}
 Δ\uz - \uz^{χ+1} &= r^{1-d}(r^{d-1}(\underline{z_{\alpha}})_{r})_{r}-(\underline{z_{\alpha}})^{\chi+1} 
 \\&= r^{1-d}(r^{d-1}(γr^{β})_{r})_{r}-(γr^{β})^{\chi+1}\\
 &= γβr^{1-d}(r^{d+β-2})_{r}-(γr^{β})^{\chi+1}\\
 &= γβ(d+β-2) r^{β-2}-(γr^{β})^{\chi+1}\\
 &= γr^{β} (β(d+β-2)r^{-2}-γ^{χ}r^{βχ})\\
 &\ge γr^{β} (β^2 r^{-2}-γ^{χ}r^{βχ})\\
 &\ge γr^{β} (β^2 R^{-2}-γ^{χ}R^{βχ})\\
 &= γr^{β}(\alpha (\uv[]^*)^{χ}R^2R^{-2} - \alpha (\uv[]^*)^{χ}R^{-βχ}R^{βχ}=0.
\end{align*}
Additionally, $\uz(x)=γ|x|^{β}\le γR^{β}=α^{\f1{χ}}\uv[]^*$.
\end{proof}

\begin{lemma}\label{lem:greaterthansubsolution}
 Let $d\ge 2$ and $\Omega\subseteq B_R(0)$. Then for any $α>0$, $z_{α}\ge \uz$ in $\Om$.
\end{lemma}
\begin{proof}
 By \eqref{eq:z} and \eqref{subsolution}, $\uz\le α^{\f1{χ}}\uv[]^*=z_{α}$ on $∂\Om\subset \overline{B_R(0)}$, and in $\Om$ we have $-Δ\uz\le -\uz^{χ+1}$, whereas $-Δz_{α}\ge -z_{α}^{χ+1}$.
\end{proof}

For $R>0$, $δ>0$ and $ω_0\in \mathbb{S}^{d-1}$ we let 
\[
 S_R(δ,ω_0)\colonequals \set{ρω \mid  ω\in \mathbb{S}^{d-1}, |ω-ω_0|<δ, ρ\in(R-δ,R) }
\]
denote a sector around $Rω$ of a spherical shell centered at $0$.

\begin{lemma}\label{LEM10b}
Let $d\ge 2$, $R>0$ and assume that 
\begin{equation}\label{domain-condition}
\text{for some } δ>0 \text{ and } ω\in \mathbb S^{d-1} \text{ we have } S=S_R(δ,ω)\subseteq \Om\subseteq B_R(0). 
\end{equation}
Then $\underline{m}(\alpha)\rightarrow \infty$ as $\alpha\rightarrow \infty$.
\end{lemma}
\begin{proof}
 Due to Lemma~\ref{lem:greaterthansubsolution}, it is sufficient to show that $\int_S \uz^{χ}\to \infty$ as $α\to \infty$. Denoting by $σ_d$ the $(d-1)$-dimensional surface area of $\overline{S_1(δ,ω)}\cap ∂B_1(0)$, we have 
 \begin{align*}
\int_{\Omega}\underline{z_{\alpha}}^{\chi}&\ge \sigma_d \int_{R-δ}^R γ^{χ}r^{χβ} r^{d-1} dr = \frac{σ_d γ^{χ}}{χβ+d} (R^{χβ+d} -(R-δ)^{χβ+d})\\
&= \frac{\sigma_{d}\alpha(\uv[]^*)^{\chi}R^{-χβ} }{\chi\alpha^{\frac{1}{2}}(\uv[]^*)^{\frac{\chi}{2}}R+d}R^{χβ+d} (1-(1-\f{δ}R)^{χβ+d})\to \infty 
\end{align*}
as $α\to\infty$.
\end{proof}

\begin{lemma}\label{lem:infty}
 If $d\in\set{1,2}$ or $\Om$ satisfies \eqref{domain-condition}, then $m_\infty=\infty$.
\end{lemma}
\begin{proof}
 From Lemma~\ref{LEM9}, we know that $m(α)\ge \underline{m}(α)$ and Lemma~\ref{LEM10a} and \ref{LEM10b} show that $\lim_{α\to\infty} \underline{m}(α)=\infty$.
\end{proof}

\begin{proof}[Proof of Theorem~\ref{THM1}]
The pair of functions $(u,v)\in (C^2(\Ombar))^2$ satisfies $\io u=m$ and \eqref{SSP} if and only if $u=αv^{χ}$, $m(α)=m$ and $v$ solves \eqref{LEM8EQ} (cf. Lemma~\ref{lem:onlyoneequation} and \eqref{def:m}). 
Unique solvability of \eqref{LEM8EQ} by Lemma~\ref{LEM8} and bijectivity of $m$ from Lemma~\ref{LEMDEFM} and  Lemma~\ref{lem:infty} thus prove Theorem~\ref{THM1}.
\end{proof}

\section{Acknowledgements}
JA acknowledges support of the National Research Foundation (NRF) of Korea (Grant No. RS-2024-00336346).


\end{document}